\documentclass{amsart}

\newcommand{\IN}{\mathbb N}
\newcommand{\IR}{\mathbb R}
\newcommand{\IZ}{\mathbb Z}
\newcommand{\IM}{\mathbb M}
\newcommand{\IT}{\mathbb T}
\newcommand{\IC}{\mathbb C}
\newcommand{\w}{\omega}
\newcommand{\e}{\varepsilon}
\newcommand{\id}{\mathrm{id}}

\newtheorem{theorem}{Theorem}

\newtheorem{corollary}{Corollary}
\theoremstyle{definition}
\newtheorem{definition}{Definition}

\newtheorem{remark}{Remark}

\title{Manifolds admitting a continuous cancellative binary operation are orientable}

\author{Taras Banakh, Igor Guran, Alex Ravsky}
\address{T.Banakh: Ivan Franko National University of Lviv (Ukraine), and Jan Kochanowski University in Kielce (Poland)}
\email{t.o.banakh@gmail.com}
\address{I.Guran: Faculty of Mechanics and Mathematics, Ivan Franko National University of Lviv,  Ukraine}
\email{igor\_guran@yahoo.com}
\address{A.Ravsky: Pidstryhach Institute for Applied Problems of Mechanics and Mathematics of National Academy of Sciences, Lviv, Ukraine}
\email{oravsky@mail.ru}
\begin{document}

\begin{abstract}
We prove that a topological manifold (possibly with boundary)  admitting a continuous cancellative binary operation is orientable. This implies that the M\"obius band admits no cancellative continuous binary operation. This answers a question posed by the second author in 2010.
\end{abstract}

\maketitle
\section*{Introduction}

It is known that the presence of a compatible algebraic structure on a topological space imposes strict restrictions on its topological structure. This phenomenon can be illustrated by the famous Adams' Theorem
\cite{Adam} saying that a sphere $S^n$ of finite dimension $n\notin\{0,1,3,7\}$ admits no continuous binary operation with two-sided unit. Each connected locally path connected space admitting a continuous semilattice operation has trivial homotopy groups (this follows from Lemma 2 in \cite{BS}). Another result of the same flavor says that a topological manifold admitting a continuous group operation is orientable, see \cite[15.19]{Lee} or \cite[15.10]{Bump}.

In this paper we shall generalize the latter result proving that a topological manifold $M$ (possibly with boundary) carrying a continuous cancellative binary operation is (isotopically) orientable. The notion of isotopically orientable topological space is introduced in Section~\ref{s1} where we prove that for topological manifolds the isotopical orientability is equivalent to the orientability in the classical sense. In Section~\ref{s2} we prove that a locally compact locally connected metrizable topological space is isotopically orientable provided it admits a continuous binary operation with open injective shifts and some extra property expressing the continuous cancellativity of the operation. Applying this criterion to topological manifolds, we  obtain our main result which says that a topological manifold admitting a continuous cancellative binary operation is orientable.

\section{Isotopically orientable manifolds}\label{s1}

In this section we introduce the notion of isotopical orientability of manifolds.

Let $E$ be a fixed topological space called a {\em model space}.
By an {\em $E$-manifold} we understand a paracompact topological space $M$ such that each point $x\in E$ has an open neighborhood $U_x$ that admits an open topological embedding $\varphi_x:U_x\to E$, called an {\em open chart}. The family of pairs $\{(U_x,\varphi_x):x\in M\}$ is called an {\em atlas} on $M$. For every points $x,y\in X$ the function
$$\varphi_{x,y}=\varphi_y\circ\varphi^{-1}_x|\varphi_x(U_x\cap U_y):\varphi_x(U_x\cap U_y)\to E$$ is called the {\em transition function}. This function is an open topological
embedding.

An $E$-manifold $M$ is defined to be isotopically {\em $E$-orientable} if it admits an atlas $\{(U_x,\varphi_x):x\in M\}$ such that for every points $x,y\in M$, any point $z\in \varphi_x(U_x)\cap \varphi_y(U_y)$ has an open neighborhood $U_z\subset \varphi_x(U_x\cap U_y)\subset E$ such that the transition map $\varphi_{x,y}|U_z:U_z\to \varphi_y(U_y)\subset E$ is open-isotopic to the identity embedding $\id:U_z\to E$. The latter means that there is a homotopy $(h_t)_{t\in[0,1]}:U_z\to E$ such that $h_0=\varphi_{x,y}|U_z$, $h_1=\id$ and each map $h_t:U_z\to E$ is an open embedding. Such homotopies will be called {\em open-isotopies}.

Each model space $E$ is an isotopically $E$-oriented $E$-manifold with the single chart $\id:E\to E$. In particular, the open M\"obius band $\IM$ is an isotopically $\IM$-oriented $\IM$-manifold. However, considered as an $\IR^2$-manifold, the M\"obius band $\IM$ is not isotopically $\IR^2$-oriented.

In order to kill such an ambiguity, we introduce the following

\begin{definition} A paracompact space $X$ is called {\em isotopically orientable} if there is a point $e\in X$ such that for any open neighborhood $E\subset X$ of $e$
 the space $X$ is a $E$-orientable $E$-manifold.
\end{definition}

Now our aim to show that for topological manifolds the isotopical orientability is equivalent to the usual orientability (defined via homologies). By a {\em topological manifold} we understand an $\IR^n_+$-manifold  modeled on the half-space
$$\IR^n_+=\{(x_i)_{i=1}^n\in\IR^n:x_n\ge0\}$$
for some $n\in\IN$. The {\em boundary} $\partial M$ of a topological manifold $M$ consists of all point $x\in M$ that has no open neighborhood homeomorphic to $\IR^n$.

An $\IR^n$-manifold $M$ is called {\em orientable} if it admits an atlas $\{(U_x,\varphi_x):x\in M\}$ such that for any $x,y\in M$ the transition function
$$\varphi_{x,y}:\varphi_{x}(U_x\cap U_y)\to\IR^n,\;\;\varphi_{x,y}:z\mapsto \varphi_y\circ\varphi^{-1}_x(z),$$ is orientation-preserving.
An embedding $f:U\to\IR^n$ of an open subset $U\subset\IR^n_+$ is called {\em orientation-preserving} if each point $z\in U$ has an open neighborhood $U_z\subset U$ such that the homomorphism in $n$-th homologies
$$H_n f:H_n(U_z,U_z\setminus\{z\};\IZ)\to H_n(\IR^n,\IR^n\setminus\{f(z)\};\IZ)$$
induced by the map $f|U_z$ coincides with the homomorphism
$$H_n s:H_n(U_z,U_z\setminus\{z\};\IZ)\to H_n(\IR^n_+,\IR^n_+\setminus\{f(z)\};\IZ)$$
induced by the shift $s:U_z\to\IR^n_+,\quad s:u\mapsto f(z)-z+u.$

\begin{theorem}\label{t2} An $\IR^n$-manifold $M$ is orientable if and only if it is isotopically orientable.
\end{theorem}

\begin{proof} To prove the ``if'' part, assume that the $\IR^n$-manifold $M$ is isotopically orientable. Then there is a point $e\in M$ such that for every open neighborhood $E\subset M$ of $e$ the space $M$ is an $E$-orientable $E$-manifold. Choose an open neighborhood $E\subset M$ of $e$ homeomorphic to $\IR^n$. To simplify notation we identify $E$ with the Euclidean space $\IR^n$.

Let $\{(U_x,\varphi_x):x\in M\}$ be an atlas witnessing that $M$ is an $E$-orientable $E$-manifold. Given any points $x,y\in M$, we should check that the transition map $\varphi_{x,y}:\varphi_x(U_x\cap U_y)\to \varphi_y(U_x\cap U_y)$ is orientation-preserving. Given any point $z\in\varphi_x(U_x\cap U_y)$ the $E$-orientability of $M$ yields an open neighborhood $U_z\subset \varphi_x(U_x\cap U_y)$ and an open-isotopy $(i_t)_{t\in[0,1]}:U_z\to E$ connecting the map $i_0=\varphi_{x,y}|U_z:U_z\to \varphi_y(U_y)\subset E$ with the identity inclusion $i_1:U_z\to \varphi_{x}(U_x)\subset E$ of $U_z$. The isotopy $(i_t)$ induces the isotopy $(j_t):U_z\to \IR^n$ defined by the formula $j_t(u)=i_t(u)+i_0(z)-i_t(z)$ for $t\in[0,1]$ and $u\in U_z$.
In fact, the isotopy $(j_t)$ is a homotopy of the pairs $(j_t):(U_z,U_z\setminus\{z\})\to (E,E\setminus\{\varphi_{x,y}(z)\})$. By the homotopical invariance of homology, the maps $\varphi_{x,y}=i_0=j_0$ and $j_1:u\mapsto u-z+\varphi_{x,y}(u)$ induce the same homomorphism in $n$-th
homologies $H_nj_0=H_nj_1:H_n(U_z,U_z\setminus\{x\})\to (E,E\setminus\{\varphi_{x,y}(z)\})$, witnessing that $\varphi_{x,y}$ is orientation-preserving and $M$ is orientable.

To prove the ``only if''part, assume that the $\IR^n$-manifold $M$ is orientable. Fix any point $e\in M$. Given any open neighborhood $E\subset M$ of $x$ we shall prove that $M$ is an $E$-orientable $E$-manifold. Replacing $E$ by a smaller open neighborhood, we can assume that $E$ is homeomorphic to $\IR^n$. To simplify notation, we shall identify $E$ with $\IR^n$. Since $M$ is an orientable $\IR^n$-manifold, there exists an atlas $\{(U_x,\varphi_x):x\in M\}$ such that for any points $x,y\in M$ the transition map $\varphi_{x,y}:\varphi_x(U_x\cap U_y)\to \varphi_y(U_x\cap U_y)$ is orientation preserving. This atlas witnesses that $M$ is an $E$-manifold. To show that this $E$-manifold is $E$-oriented, for any points $x,y\in M$ and $z\in \varphi_x(U_x\cap U_y)$ we should find a neighborhood $U_z\subset\varphi_x(U_x\cap U_y)$ such that the map $\varphi_{x,y}|U_z:U_z\to \varphi_y(U_x\cap U_y)\subset E$ is open-isotopic to the identity embedding $U_z\subset\varphi_x(U_x\cap U_y)\subset E$.
Choose  two open neighborhoods $U_z\subset W_z$ of $z$ in $\varphi_x(U_x\cap U_y)$ such that the pair $(W_z,\overline{U}_z)$ is homeomorphic to the pair of balls $(B(2),\overline{B}(1))$ in $\IR^n$.
Here for $r>0$ by $B(r)=\{a\in\IR^n:\|a\|<r\}$ and $\overline{B}(r)=\{a\in\IR^n:\|a\|\le r\}$ we denote the open and closed $r$-balls centered at the origin. It follows that the boundaries $\partial U_z$ and $\varphi_{x,y}(\partial U_z)$ are bicollared topological $(n-1)$-spheres in $\IR^n$. Choose a positive real number $R$ so large that $\overline{U_z}\cup\varphi_{x,y}(\overline{U_z})\subset B(R)$.
The Annulus Theorem (proved by combined efforts of  Rad\'o \cite{Rado}, Moise \cite{Moise}, Kirby \cite{Kirby} and Quinn \cite{Quinn}), allows us to extend the homeomorphism $\varphi_{x,y}|\overline{U}_z$ to a homeomorphism $h$ of the closed ball $\overline{B}(R)$. Since the embedding $\varphi_{x,y}|U_z:U_z\to B(R)$ preserves the orientation, the homeomorphism $h$ preserves the orientation of $\overline{B}(R)$ and the restriction $h|S(R)$ preserves the orientation of the sphere $S(R)=\overline{B}(R)\setminus B(R)$. By Brown-Gluck Theorem \cite{BG}, the Annulus Theorem implies the Stable Homeomorphism Theorem, which implies that the homeomorphism group of the sphere $S(R)$ has two connected components and hence the orientation-preserving homeomorphism $h|S(R)$ is isotopic to the identity homeomorphism. This fact can be used to show that the homeomorphism $h$ can be replaced by a homeomorphism of $\bar B(R)$, which coincides with $\varphi_{x,y}|\overline{U}_z$ on $\overline{U}_z$ and is identity on the sphere $S(R)$. Now the Alexander trick \cite{Alex} (see also Lemma 5.6 in \cite{Hilden}) guarantees that $h$ is isotopic to the identity homeomorphism of $\bar B(R)$ and hence $\varphi_{x,y}|U_z=h|U_z$ is open-isotopic to the identity embedding $i:U_z\to E=\IR$. This means that the $E$-manifold $M$ is $E$-orientable.
\end{proof}

An $\IR^n_+$-manifold $M$ is defined to be {\em orientable} if its interior $M\setminus \partial M$ is an orientable $\IR^n$-manifold.  Theorem~\ref{t2} implies:

\begin{corollary} An $\IR^n_+$-manifold $M$ is orientable if and only if its interior $M\setminus\partial M$  is isotopically orientable.
\end{corollary}

\section{Isotopical orientability of spaces and manifolds with a compatible algebraic structure}\label{s2}

Let us recall that a map $f:X\to Y$ between topological spaces is {\em open} if for every open set $U\subset X$ the image $f(U)$ is open in $Y$. It is clear that each  injective open continuous map is a topological embedding.

\begin{theorem}\label{E} A locally compact locally connected metrizable space $S$ is isotopically orientable if $S$ admits a continuous binary operation $\cdot:S\times S\to S$ and a point $e\in S$ such that
\begin{enumerate}
\item the right shift $\rho_e:S\to S$, $\rho_e:x\mapsto xe$, is injective and open;
\item for every $a\in S$ the left shift $\lambda_a:S\to S$, $\lambda_a:x\mapsto ax$, is injective and open;
\item for every $a\in S$ and neighborhood $O_e\subset S$ of $e$ there is a neighborhood $U_a\subset S$ of $a$ such that $\bigcap_{x\in U_a}xO_e$ is a neighborhood of the point $ae$ in $S$.
\end{enumerate}
\end{theorem}

\begin{proof} Since the space $S$ is locally connected, the connected components of $S$ are closed-and-open in $S$.
Now we see that it suffices to prove that each connected component $X$ of $S$ is an $E$-orientable $E$-manifold for any open neighborhood $E$ of the point $e$.

The metrizable space $X$, being connected and locally compact, can be written as the countable union $X=\bigcup_{n\in\w}K_n$ of compact subsets $K_n\subset X$ such that $K_0=\emptyset$ and  each $K_n$ lies in the interior $K_{n+1}^\circ$ of $K_{n+1}$ in $X$.

By our hypothesis, the right shift $\rho_e:X\to S$ is an open embedding. Since the image
$Xe=\rho_e(X)$ is homeomorphic to $X$, it suffices to check that $Xe$ is an $E$-orientable $E$-manifold. %Each point of the space $Xe$ can be uniquely written as $xe=\rho_e(x)$ for some point $x\in X$.

Fix a metric $d$ generating the topology of the space $S$. For a point $x\in S$ and a real number $r>0$ denote by $B_x(r)=\{y\in S:d(x,y)<r\}$ the open $r$-ball centered at $x$ and let $\overline{B}_x(r)$ be the closure of $B_x(r)$ in $S$. The $r$-balls $B_e(r)$ and $\overline{B}_e(r)$ centered at the point $e$  will be denoted by $B(r)$ and $\bar B(r)$, respectively.

Find $\e_0>0$ such that the ball $B(\e_0)$ has compact closure in the neighborhood $E$ of $e$. Using the condition (3) of the theorem and the compactness of the sets $K_n$ and $K_{n}\setminus K_{n-1}^\circ$, $n\in\IN$, it is easy to construct a decreasing sequence of positive real numbers $(\e_n)_{n\in\IN}$ such that for every $n\in\IN$ the following conditions are satisfied:
\begin{itemize}
\item[(a)] $\overline{B}_x(\e_n)\subset K_{n+1}^\circ\setminus K_{n-2}$ for every $x\in K_n\setminus K_{n-1}^\circ$;
\item[(b)] $x\cdot\overline{B}(\e_n)\subset (K_{n+1}^\circ\setminus K_{n-2})\cdot e$ for every $x\in K_n\setminus K_{n-1}^\circ$;
\item[(c)] $x\cdot \overline{B}(\e_n)\subset \bigcap_{y\in B_x(\e_n)}y\cdot B(\e_0)$
for every $x\in K_n$;
\item[(d)] for every $x,y\in K_n$ with  $x\cdot \overline{B}(\e_n)\cap y\cdot\overline{B}(\e_n)\ne\emptyset$, we get $d(x,y)<\e_{n-1}$.
\item[(e)] for every points $x,y\in K_n$ with $d(x,y)<\e_n$ there is a continuous path $\pi:[0,1]\to X$ of diameter $<\e_{n-1}$ such that $\pi(0)=x$ and $\pi(1)=y$.
\end{itemize}

Now we shall construct an atlas $\{(U_{xe},\varphi_{xe}):xe\in Xe\}$ witnessing that $Xe$ is an $E$-orientable $E$-manifold. %We recall that each point of the space $Xe$ can be uniquely written as $xe=\rho_e(x)$ {\bf Znovu} for some $x\in X$.

For every point $x\in X$ find $n_x\in\IN$ with $x\in K_{n_x}\setminus K_{n_x-1}$ (we recall that $K_0=\emptyset$).
The property (b) guarantees that $x\cdot \overline{B}(\e_{n_x+5})\subset x\cdot \overline{B}(\e_{n_x})\subset K_{n_x+1}\cdot e\subset Xe$ and hence $$U_{xe}:=x\cdot B(\e_{n_x+5})=\lambda_x(B(\e_{n_x+5}))\subset Xe$$ is an open neighborhood of the point $xe$ in $Xe$.
Define the chart $\varphi_{xe}:U_{xe}\to B(\e_{n_x+5})\subset E$ by the formula
$\varphi_{xe}=\lambda_x^{-1}|U_{xe}$ and observe that the map $\varphi_{xe}:U_{xe}\to E$ is an open embedding being the inverse map to the open embedding $\lambda_x|B(\e_{n_x+5}):B(\e_{n_x+5})\to U_{xe}\subset Xe$.

We claim that the atlas $\{(U_{xe},\varphi_{xe}):xe\in Xe\}$ witnesses that $Xe$ is an $E$-oriented $E$-manifold. Take any two points $xe,ye\in Xe$ such that $U_{xe}\cap U_{ye}\ne\emptyset$ and consider the transition function
$$\varphi_{xe,ye}:\varphi_{xe}(U_{xe}\cap U_{ye})\to \varphi_{ye}(U_{xe}\cap U_{ye})\subset E,\quad \varphi_{xe,ye}:z\mapsto \varphi_{ye}\circ\varphi^{-1}_{xe}(z)=\lambda_{y}^{-1}\circ\lambda_{x}(z)$$ defined on the open subset $\varphi_{xe}(U_{xe}\cap U_{ye})$ of $E$. Given any point $z\in \varphi_{xe}(U_{xe}\cap U_{ye})$, we need to find a neighborhood $U_z\subset \varphi_{xe}(U_{xe}\cap U_{ye})$ such that the map $\varphi_{xe,ye}|U_z:U_z\to E$ is open-isotopic to the identity embedding $\id:U_z\to E$. Fix any neighborhood $U_z$ of $z$ with compact closure $\bar U_z$ in $\varphi_{xe}(U_{xe}\cap U_{ye})$.

Let $n_x,n_y\in\IN$ be the unique numbers such that $x\in K_{n_x}\setminus K_{n_x-1}$ and $y\in K_{n_y}\setminus K_{n_y-1}$.  We claim that $|n_x-n_y|\le 2$.
The property (b) guarantees $$U_{xe}=x\cdot B(\e_{n_x+5})\subset x\cdot B(\e_{n_x})\subset(K_{n_x+1}\setminus K_{n_x-2})\cdot e\mbox{ \ and \ }U_{ye}\subset (K_{n_y+1}\setminus K_{n_y+2})\cdot e.$$
Since $U_{xe}\cap U_{ye}\ne\emptyset$, the injectivity of the shift $\rho_e$ implies
 $$(K_{n_x+1}\setminus K_{n_x-2})\cap (K_{n_y+1}\setminus K_{n_y-2})\ne\emptyset,$$which yields $|n_x-n_y|\le 2$.
Then for the number $n=\max\{n_x,n_y\}$ we get $x,y\in K_n\setminus K_{n-3}$.
Also $U_{xe}=x\cdot B(\e_{n_x+5})\subset x\cdot B(\e_{n+3})$.

Since $\emptyset \ne U_{xe}\cap U_{ye}=x{\cdot} B(\e_{n_x+5})\cap y{\cdot} B(\e_{n_x+5})\subset x{\cdot}B(\e_{n+3})\cap y{\cdot}B(\e_{n+3})$, the item (d) ensures that $d(x,y)<\e_{n+2}$. By the item (e), the points $x,y$ can be linked by a continuous path $\pi:[0,1]\to S$ of diameter $<\e_{n+1}$ such that $\pi(0)=x$, $\pi(1)=y$.

Now we are able to prove that the map $\varphi_{xe,ye}|U_z:U_z\to E$ is open-isotopic to the identity embedding $U_z\to E$.
For this consider the continuous map $g:[0,1]\times \overline{B}(\e_0)\to [0,1]\times S$, $g:(t,z)\mapsto (t,\pi(t)\cdot z)$, defined on the compact set $[0,1]\times\overline{B}(\e_0)$. The injectivity of the left shifts on $S$ implies that the map $g$ is injective and hence is a topological embedding.

The choice of the path $\pi$ guarantees that  $\pi([0,1])\subset B_x(\e_{n+1})$. Then condition (c) implies that
$$x\overline{U}_{z}= \varphi^{-1}_{xe}(\overline U_z)\subset U_{xe}\subset x\cdot B(\e_{n+1})\subset \bigcap_{t\in[0,1]}\pi(t)\cdot B(\e_0)=\bigcap_{t\in[0,1]}\lambda_{\pi(t)}(B(\e_0)).$$ Consequently, the map $g^{-1}|[0,1]\times x\overline{U}_z:[0,1]\times x\overline{U}_z\to [0,1]\times\overline{B}(\e_0)$ is a topological embedding and the map $h:[0,1]\times \overline{U}_z\to [0,1]\times\overline{B}(\e_0)$, \ $h:(t,u)\mapsto (t,\lambda_{\pi(t)}^{-1}{\circ}\lambda_x(u))$, is a well-defined topological embedding of $[0,1]\times \overline{U}_z$ into $[0,1]\times \overline{B}(\e_0)\subset [0,1]\times E$. Since for every $t\in[0,1]$ the map $\lambda_{\pi(t)}^{-1}{\circ}\lambda_x|U_z:U_z\to E$ is an open embedding, the family $\big(\lambda^{-1}_{\pi(t)}{\circ}\lambda_x|U_z)_{t\in[0,1]}$ is an open-isotopy linking the transition map $\varphi_{xe,ye}=\lambda_{\pi(1)}^{-1}{\circ}\lambda_x|U_z$ with the identity map $\lambda_{\pi(0)}^{-1}{\circ}\lambda_x|U_z$ of $U_z$.
\end{proof}

Applying Theorems~\ref{t2} and \ref{E} to topological manifolds, we get the following theorem.

\begin{theorem}\label{man} A topological manifold $M$ is orientable provided $M$ admits a continuous binary operation\break $\cdot :M\times M\to M$ such that
\begin{enumerate}
\item for some $e\in M\setminus\partial M$ the right shift $\rho_e:M\to M$ is injective;
\item for every $x\in M\setminus\partial M$ the left shift $\lambda_x:M\to M$ is injective.
\end{enumerate}
\end{theorem}

\begin{proof} To show that an $\IR^n_+$-manifold $M$ is orientable, we need to check that its interior $N=M\setminus\partial M$ is orientable.

By the Open Domain Principle, each continuous injective map $f:U\to\IR^n$ defined on an open subset $U\subset \IR^n$ is an open topological embedding. This implies that  the shifts $\rho_e:N\to M$ and $\lambda_x:N\to M$, $x\in N$, are open topological embeddings. Moreover, $\rho_e(N)\subset N$ and $\lambda_x(N)\subset N$ for all $x\in N$, which means that the interior $N$ of $M$ is closed under the binary operation. Now we see that the binary operation restricted to $N$ satisfies the conditions (1) and (2) of Theorem~\ref{E}. It remains to check the condition (3). Fix any neighborhood $O_e\subset N$ of the point $e$. Given any point $x\in N$, we should find a neighborhood $U_x\subset N$ such that $\bigcap_{u\in U_x}uO_e$ is a neighborhood of $x$. Find a neighborhood $O_{xe}\subset N$ of the point $xe\in N$, homeomorphic to $\IR^n$.
Choose a neighborhood $U_e\subset N$ of $e$ that has compact closure $\bar U_e$ in $O_e$ such that  $x\cdot\bar U_e\subset O_{xe}$ and $\bar U_e$ is homeomorphic to closed unit ball $\overline{B}=\{x\in\IR^n:\|x\|\le 1\}$ in $\IR^n$. Fix a homeomorphism $\psi:x\bar U_e\to \overline{B}$ such that $\psi(xe)=\vec 0$. Using the compactness of $\bar U_e$ and the continuity of the binary operation, we can find a neighborhood $U_x\subset N$ of $x$ such that $\|\psi(u\cdot z)-\psi(x\cdot z)\|<1$ for all $u\in U_x$ and $z\in \bar U_e$. Since $xe\notin x\cdot\partial U_e$, we can replace $U_x$ by a smaller neighborhood and additionally assume that $xe\notin \bar U_x\cdot \partial U_e$.

We claim that for every $u\in U_x$ the set $u\cdot \bar U_e$ contains the point $xe$. For this consider the homeomorphism $h:\bar U_e\to \bar B$, $h:z\mapsto \psi(x\cdot z)$, and the continuous map $f:\bar U_e\to \bar B$, $f:z\mapsto \psi(x\cdot z)-\psi(u\cdot z)=h(z)-\psi(u\cdot z)$. The Brouwer Fixed Point Theorem \cite[7.3.18]{En} guarantees that the self-map $g=f\circ h^{-1}:\bar B\to\bar B$ of the closed ball $\bar B$ has a fixed point $b\in\bar B$. Find a unique point $z\in\bar U_e$ such that $h(z)=b$ and observe that $h(z)=b=g(b)=f\circ h^{-1}(b)=f(z)=h(z)-\psi(u\cdot z)$ and hence $\psi(u\cdot z)=\vec 0=\psi(xe)$, which implies that $xe=uz\in u\bar U_e$. Since $xe\notin \bar U_x\cdot \partial U_e$, we can choose a connected neighborhood $W_x\subset U_x$ of $x$ such that $W_xe\cap \bar U_x\cdot\partial O_e=\emptyset$. We claim that  $W_xe\subset \bigcap_{u\in U_x}uO_e$. To derive a contradiction, assume that some point $w\in W_x$ is not contained in the set $uO_e$ for some $u\in U_x$. It follows that $u\cdot \partial U_e$ is a topological copy of the $(n-1)$-dimensional sphere in the topological copy $O_{xe}$ of the space $\IR^n$. By Jordan-Brouwer Separation Theorem \cite[2B.1]{Hat}, the sphere $u\cdot\partial U_e$ separates $O_{xe}$ into two connected components: $uU_e$ and $O_{xe}\setminus u\bar U_e$. Taking into account that the set $W_xe$ is connected and meets both sets $u\bar U_e\ni xe$ and $O_{xe}\setminus uO_e\subset O_{xe}\setminus u\bar U_e$, we see that $W_xe$ meets the set $u\cdot \partial U_e$, which contradicts the choice of $W_x$. This contradiction shows that $\bigcap_{w\in U_x}uO_e$ is a neighborhood of $x$ in $O_{xe}$. Now it is legal to apply Theorem~\ref{E} and conclude that the $\IR^n$-manifold $N$ is isotopically orientable. By Theorem~\ref{t2}, it is orientable and then the topological manifold $M=N\cup\partial M$ is orientable as well.
\end{proof}

A binary operation $\cdot:X\times X\to X$ on a set $X$ is called {\em left cancellative} (resp. {\em right cancellative}) if for any points $x,y,z\in X$ the equality $zx=zy$ (resp. $xz=yz$) implies $x=y$.
This is equivalent to saying that for every $z\in X$ the left shift $\lambda_z:X\to X$ (resp. the right shift $\rho_z:X\to X$) is injective. A binary operation is {\em cancellative} if it is both left and right cancellative. The structure of cancellative topological semigroups on manifolds was studied by Brown, Houston \cite{BH} and Hofmann, Weiss \cite{HW}.

Applying Theorem~\ref{man} to cancellative operations, we get the following generalization of the well-known result on orientability of Lie groups.

\begin{corollary}\label{c1} Each topological manifold admitting a cancellative continuous binary operation is orientable.
\end{corollary}

We say that an element $e\in X$ is a {\em right unit} for a binary operation on a set $X$ is $xe=x$ for all $x\in X$.

 \begin{corollary}\label{c2} A topological manifold $M$ is orientable if it admits a left cancellative continuous binary operation possessing a right unit $e\in M\notin\partial M$.
\end{corollary}

\begin{remark} Corollary~\ref{c1} implies that the M\"obius band $\IM$ admits no  cancellative continuous binary operation. Yet, $\IM$ has a natural structure of an abelian topological inverse monoid, see \cite{HM}. To construct such a structure on $\IM$, consider the unit circle $\IT=\{z\in\IC:|z|=1\}$ endowed with the operation of multiplication of complex numbers and the unit interval $[0,1]$ endowed with the operation of minimum. Then the product $\IT\times [0,1]$ is a commutative compact topological inverse monoid, $(1,1)$ is the unit of $\IT\times[0,1]$, and $\IT\times\{0\}$ is the minimal ideal of $\IT\times[0,1]$. On $\IT\times[0,1]$ consider the congruence $\sim$ identifying the points $(z,0)$ and $(-z,0)$ for $z\in \IT$. Then the quotient semigroup $\IM=\IT\times[0,1]$ is homeomorphic to the M\"obius band and is a commutative compact topological inverse monoid.
\end{remark}

\end{document}